\documentclass{amsart2010}
\usepackage{graphicx}
\usepackage{amsmath}
\usepackage{amsfonts}
\usepackage{amssymb}
\usepackage{booktabs}

\newtheorem{thm}{Theorem}
\newtheorem{cor}[thm]{Corollary}
\newtheorem{lem}[thm]{Lemma}
\newtheorem{prop}[thm]{Proposition}
\theoremstyle{definition}

\newtheorem*{theorem*}{Theorem}

\theoremstyle{remark}

\numberwithin{equation}{section}

\newcommand{\To}{\longrightarrow}

\def\<{\langle}
\def\>{\rangle}
\begin{document}
	
\title[]{Complex symmetry and cyclicity of composition operators on $H^2(\mathbb{C}_+)$}

\author{S. Waleed Noor and Osmar R. Severiano}%

\address{IMECC, Universidade Estadual de Campinas, Campinas-SP, Brazil.}

\email{$\mathrm{waleed@ime.unicamp.br}$(1st author),$\mathrm{osmar.rrseveriano@gmail.com}$(2nd author).}

\begin{abstract} In this article, we completely characterize the complex symmetry, cyclicity and hypercyclicity of composition operators $C_\phi f=f\circ\phi$ induced by affine self-maps $\phi$ of the right half-plane $\mathbb{C}_+$ on the Hardy-Hilbert space $H^2(\mathbb{C}_+)$. The interplay between complex symmetry and cyclicity plays a key role in the analysis. We also provide new proofs for the normal, self-adjoint and unitary cases and for an adjoint formula discovered by Gallardo-Guti\'{e}rrez and Montes-Rodr\'{i}guez.
\end{abstract}

\maketitle{}

\section*{Introduction}

A bounded operator $T$ on a separable Hilbert space $\mathcal{H}$ is \emph{complex symmetric} if there exists an orthonormal basis for $\mathcal{H}$ with respect to which $T$ has a self-transpose matrix representation. An equivalent definition also exists. A \emph{conjugation} is a conjugate-linear operator $C:\mathcal{H}\to\mathcal{H}$ that satisfies the conditions \\

(a) $C$ is \emph{isometric}: $\langle Cf,Cg \rangle=\langle g,f \rangle$ $\forall$ $f,g\in\mathcal{H}$,\\

(b) $C$ is \emph{involutive}: $C^2=I$.\\ \\
We say that $T$ is $C$-\emph{symmetric} if $CT=T^*C$, and complex symmetric if there exists a conjugation $C$ with respect to which $T$ is $C$-symmetric.

Complex symmetric operators on Hilbert spaces are natural generalizations of complex symmetric matrices, and their general study was initiated by
Garcia, Putinar, and Wogen (\cite{Garc2},\cite{Garc3},\cite{Wogen1},\cite{Wogen2}). The class of complex symmetric operators includes a large number of concrete examples including all normal operators.

An operator $T$ on $\mathcal{H}$ is said to be \emph{cyclic} if there exists a vector $f\in\mathcal{H}$ for which the linear span of its orbit $(T^nf)_{n\in\mathbb{N}}$ is dense in $\mathcal{H}$. If the orbit $(T^nf)_{n\in\mathbb{N}}$ itself is dense in $\mathcal{H}$, then $T$ is said to be \emph{hypercyclic}. In these cases $f$ is called a \emph{cyclic} or \emph{hypercyclic} vector for $T$ respectively.
If we assume that $T$ is both complex symmetric and cyclic (hypercyclic), then the relation $CT=T^*C$ implies that $T^*$ must also be cyclic (hypercyclic). The conjugation $C$ acts as a bijection between cyclic (hypercyclic) vectors of $T$ and $T^*$. Two monographs \cite{Bayart-Matheron} and \cite{Linear Chaos} on the  dynamics of linear operators have appeared recently. 

If $X$ is a Banach space of holomorphic functions on an open set $U\subset\mathbb{C}$ and if $\phi$ is a holomorphic self-map of $U$, the \emph{composition operator} with \emph{symbol} $\phi$ is defined by $C_\phi f=f\circ\phi$ for any $f\in X$. The emphasis here is on the comparison of  properties of $C_\phi$ with those of the symbol $\phi$. 

If $X$ is the Hardy space $H^2(\mathbb{C}_+)$ of the open right half-plane $\mathbb{C}_+$, then a holomorphic self-map $\phi$ of $\mathbb{C}_+$ induces a bounded  $C_\phi$ on $H^2(\mathbb{C}_+)$ if and only if $\phi$ has a finite angular derivative at the fixed point $\infty$. That is, if $\phi(\infty)=\infty$ and if the non-tangential limit
\begin{equation}\label{Ang. Der. at infinity}
\phi'(\infty):=\lim_{w\to\infty}\frac{w}{\phi(w)}
\end{equation}
exists and is finite. This was proved by Matache in \cite{Matache weighted}. Then Elliott and Jury \cite{Eliot Jury} prove that the norm of $C_\phi$ on $H^2(\mathbb{C}_+)$ is given by $|| C_{\phi}||=\sqrt{{\phi'}(\infty)}$.
Matache \cite{Matache comp halfplane} also showed that the only  linear fractional selfmaps of $\mathbb{C}_+$ that induce bounded composition operators on $H^2(\mathbb{C}_+)$ are the \emph{affine maps}

\begin{equation}\label{LFT symbol}
\phi(w)=aw+b
\end{equation}
where $a>0$ and Re$(b)\geq 0$. In this case, $C_\phi$ is normal on $H^2(\mathbb{C}_+)$ if and only if $\phi(w)=aw+b$ with $a=1$ or Re$(b)= 0$. This was first proved by Gallardo-Guti\'{e}rrez and Montes-Rodr\'{i}guez \cite{Eva Adjoints Dirichlet} and then again with a different proof by Matache \cite{Matache Inv and normal}. The authors of \cite{Eva Adjoints Dirichlet} work with the upper half-plane $\Pi$ whereas \cite{Matache Inv and normal} is concerned with the right half-plane $\mathbb{C}_+$. Hence the necessary translation of results must be made.
The study of complex symmetry of composition operators on the Hardy-Hilbert space of the unit disk $H^2(\mathbb{D})$ was initiated by Garcia and Hammond \cite{Garc1}. They showed that involutive disk automorphisms induce \emph{non-normal} complex symmetric composition operators. Then Narayan, Sievewright and Thompson \cite{Narayan Thompson CS CO} discovered non-automorphic symbols with the same property. The general problem in the disk case is far from being solved.  On the other hand the cyclity and hypercyclicity phenomena  for composition operators in the linear fractional disk case have been characterized (see \cite{Bourdon-Shapiro Cyclic Phenomena} and \cite{Eva role of spectrum}). The objective here is to characterize the complex symmetry, cyclicity  and hypercyclicity of $C_\phi$ in the linear fractional half-plane case. The interplay between complex symmetry and linear dynamics will play a key role in our analysis.

The plan of the paper is as follows. In Section 1, after some preliminaries, we provide a different proof of the adjoint formula for linear fractional composition operators first discovered by Gallardo-Guti\'{e}rrez and Montes-Rodr\'{i}guez \cite{Eva Adjoints Dirichlet}. This is used to give new and shorter proofs for the normal, self-adjoint and unitary cases. In Section 2 we characterize complex symmetry of $C_\phi$ on $H^2(\mathbb{C}_+)$. In particular we show that these are precisely the normal ones. In Section 3 we consider the cyclicity of $C_\phi$ proving that this occurs only when $\phi$ is a non-automorphism with no fixed points in $\mathbb{C}_+$.  Finally in Section 4 we prove that $H^2(\mathbb{C}_+)$ supports no hypercyclic  linear fractional $C_\phi$. Our main results are summarized in the following table.  \\

\begin{tabular}{@{}|l|l|l|l|@{}}
	\toprule
	Symbol $\phi(w)=aw+b$                              & Comp. Symmetric $C_\phi$& Cyclic $C_\phi$& Hypercyclic $C_\phi$\\ \midrule
	$\mathrm{Re}(b)=0$                          &    \ \ \ \ \ \ \ \  \ \ \ \ \  $\checkmark$     &    \ \ \ \ \    \text{\sffamily X}    &       \ \ \ \ \ \ \ \    \text{\sffamily X}           \\ \midrule
	$a=1$ \& $\mathrm{Re}(b)>0$          &        \ \ \ \ \ \ \ \  \ \ \ \ \  $\checkmark$          &    \ \ \ \ \   $\checkmark$      &        \ \ \ \ \ \ \ \    \text{\sffamily X}           \\ \midrule
	$a<1$ \& $\mathrm{Re}(b)>0$  &   \ \ \ \ \ \ \ \  \ \ \ \ \   \text{\sffamily X}          &      \ \ \ \ \    \text{\sffamily X}      &         \ \ \ \ \ \ \ \    \text{\sffamily X}          \\ \midrule
	$a>1$ \& $\mathrm{Re}(b)>0$ &    \ \ \  \ \ \ \ \  \ \ \ \ \     \text{\sffamily X}             &     \ \ \ \ \    $\checkmark$    &           \ \ \ \ \ \ \ \    \text{\sffamily X}        \\ \bottomrule
\end{tabular} \\ \\
 \section{Preliminaries}
\subsection{The Hardy space $H^2(\mathbb{C}_+)$} Let $\mathbb{C}_+$ be the open right half-plane. The Hardy space $H^2(\mathbb{C}_+)$ is the Hilbert space of analytic functions on $\mathbb{C}_+$ for which the norm 
 \[
 ||f||^2_2=\sup_{0<x<\infty}\int_{-\infty}^\infty|f(x+iy)|^2dy
 \]
is finite. For each $\alpha\in\mathbb{C}_+$, let $k_\alpha$ denote the \emph{reproducing kernel} for $H^2(\mathbb{C}_+)$ at $\alpha$; that is,
\[
k_\alpha(w)=\frac{1}{w+\bar{\alpha}}.
\]
These kernels satisfy the fundamental relation $\langle f,k_\alpha\rangle=f(\alpha)$ for all $f\in H^2(\mathbb{C}_+)$. If $\phi$ is a holomorphic self-map of $\mathbb{C}_+$, then a simple computation gives
\begin{equation}\label{eq: adjoint on kernel}
C_\phi^*k_\alpha=k_{\phi(\alpha)}
\end{equation}
for each $\alpha\in\mathbb{C}_+$.		
		
\subsection{Affine composition operators} The linear fractional self-maps $\phi$ of $\mathbb{C}_+$ that induce bounded composition operators on $H^2(\mathbb{C}_+)$ are the affine maps
\begin{equation}\label{LFT symbol2}
\phi(w)=aw+b
\end{equation}
where $a>0$ and Re$(b)\geq 0$. Such a map $\phi$ is said to be of \emph{parabolic type} if $a=1$ and is a \emph{parabolic automorphism} if additionally Re$(b)= 0$. Similarly $\phi$ is of \emph{hyperbolic type} if $a\neq 1$ and is a \emph{hyperbolic automorphism} if additionally Re$(b)=0$. Gallardo-Guti\'{e}rrez and Montes-Rodr\'{i}guez \cite[Theorem 7.1]{Eva Adjoints Dirichlet} proved a formula for the adjoint of such affine composition operators. We provide a short proof of this result.

\begin{prop} \label{LFT Adjoint}If $\phi$ is as in \eqref{LFT symbol2}, then $C_\phi^*=a^{-1}C_\psi$, where $
	\psi(w)=a^{-1}w+a^{-1}\bar{b}.
	$
	\end{prop}
		\begin{proof} We observe that for each $w\in \mathbb{C}^{+}$, we have
		\begin{align*}
			(C_{\phi}k_\alpha)(w)&=
			\frac{1}{a w+b+\overline{\alpha}}
			=\frac{1}{a \left( w+a^{-1}\overline{\alpha}+a^{-1} b\right) }
			= \frac{1}{a\left( w+\overline{\psi(\alpha)}\right) } \\
			&=a^{-1}k_{\psi(\alpha)}(w)=(a^{-1}C^*_\psi k_\alpha)(w).
		\end{align*}	
			The completeness of the reproducing kernels $(k_\alpha)_{\alpha\in\mathbb{C}_+}$ in $H^2(\mathbb{C}_{+})$ implies that $C_\phi=a^{-1}C^*_\psi$, or equivalently $C_\phi^*=a^{-1}C_\psi$.
			\end{proof}
		
		Proposition \ref{LFT Adjoint} allows us to obtain new and shorter proofs for the normal, self-adjoint and unitary composition operators (see Theorems 2.4, 3.1 and 3.4 of \cite{Matache Inv and normal}).
	
	\begin{thm} \label{normal, unitary, self-adjoint}Let $\phi(w)=aw+b$ with $a>0$ and Re$(b)\geq 0$. Then 
		\begin{enumerate}
		\item $C_\phi$ is normal if and only if $a=1$ or Re$(b)=0$,
		\item $C_\phi$ is self-adjoint if and only if $a=1$ and $b\geq 0$, 
		\item $C_\phi$ is unitary if and only if $a=1$ and Re$(b)=0$.
		\end{enumerate}
		\end{thm}
	\begin{proof}By Proposition \ref{LFT Adjoint}, the operator $C_\phi$ is normal if and only if $C_\phi C_\psi=C_\psi C_\phi$. This is equivalent to the equality $\phi\circ\psi=\psi\circ \phi$. For $w\in\mathbb{C}_+$, we have
		\begin{equation}\label{phi o psi}
		(\phi\circ\psi)(w)=a(a^{-1}w+a^{-1}\bar{b})+b=w+2\mathrm{Re}(b)
		\end{equation}
		and similarly
		\[
		(\psi\circ\phi)(w)=a^{-1}(aw+b)+a^{-1}\bar{b}=w+2a^{-1}\mathrm{Re}(b).
		\]
		Therefore $\phi\circ\psi=\psi\circ \phi$  $\Longleftrightarrow$ $(1-a^{-1})\mathrm{Re}(b)=0$ $\Longleftrightarrow$ $a=1$ or Re$(b)=0$. Similarly $C_\phi$ is self-adjoint $\Longleftrightarrow$ $C_\phi=C_\phi^*=a^{-1}C_\psi$. If we apply this operator equality to the reproducing kernel $k_1(w)=\frac{1}{w+1}$, we get 
		\begin{align*}
		&\frac{1}{\phi(w)+1}=\frac{a^{-1}}{\psi(w)+1} \Longleftrightarrow \psi(w)+1=a^{-1}\phi(w)+a^{-1} \\ 
		&\Longleftrightarrow a^{-1}w+a^{-1}\bar{b}+1=w+a^{-1}b+a^{-1}.
		\end{align*}
		The last equality holds precisely when $a=1$ and $b\geq 0$. Finally if $C_\phi $ is unitary then $a^{-1}C_\phi C_\psi=I=a^{-1}C_\psi C_\phi$ and in particular $C_{\phi\circ\psi}=aI$. Applying the latter identity to $k_1$ and using \eqref{phi o psi} gives
		\[
		\frac{1}{w+2\mathrm{Re}(b)+1}=\frac{a}{w+1}\Longleftrightarrow w+1=aw+2a\mathrm{Re}(b)+a
		\]
		which clearly holds precisely when $a=1$ and Re$(b)=0$.
		\end{proof}
	
	In the next section we address the first main theme of this work which is the complex symmetry of composition operators on $H^2(\mathbb{C}_+)$.
	
	\section{Complex Symmetry of $C_\phi$} According to Theorem \ref{normal, unitary, self-adjoint}, an affine composition operator $C_\phi$ on $H^2(\mathbb{C}_+)$ is normal if and only if $\phi$ is an automorphism or a parabolic non-automorphism. Since normal operators are complex symmetric, it follows that to characterize all composition operators that are complex symmetric we must consider the hyperbolic \emph{non-automorphisms}. These are precisely the symbols 
	\[
	\phi(w)=aw+b \ \ \mathrm{with}  \  a\in (0,1)\cup(1,\infty) \ \mathrm{and}  \  \mathrm{Re}(b)>0.
	\]
	 In this case we shall say $\phi$ is of \emph{type $\mathrm{I}$} if $a\in(0,1)$ and of \emph{type} $\mathrm{II}$ if $a\in(1,\infty)$. The corresponding composition operators $C_\phi$ shall also be called type $\mathrm{I}$ and $\mathrm{II}$ respectively. Note that according to Proposition \ref{LFT Adjoint} the adjoint of each $C_\phi$ of type $\mathrm{I}$ is a  scalar multiple of a $C_\psi$ of type $\mathrm{II}$ and vice versa. Hence if one of them is both  complex symmetric and cyclic, then so must be the other. So  to show that $C_\phi$ is \emph{not} complex symmetric when $\phi$ is a hyperbolic non-automorphism, it is enough to prove that type $\mathrm{I}$ composition operators are \emph{not} cyclic whereas those of type $\mathrm{II}$ \emph{are} cyclic.
	 \subsection{Symbols of type $\mathrm{I}$} Bourdon and Shapiro \cite[Proposition 2.7]{Bourdon-Shapiro Cyclic Phenomena}  proved that if the adjoint of a bounded linear operator $T$ on a Hilbert space has a multiple eigenvalue, then $T$ is not cyclic. Let $\psi(w)=aw+b$ with $a\in(1,\infty)$ and Re$(b)>0$, in which case $\psi$ is of type $\mathrm{II}$. For each complex $\lambda$ define the function
	 \[
	 f_\lambda(w)=\left(w+\frac{b}{a-1}\right)^\lambda
	 \]
	 which is holomorphic in $\mathbb{C}_+$  since $\mathrm{Re}(\frac{b}{a-1})>0$, and $f_\lambda\in H^2(\mathbb{C}_+)$ if and only if Re$(\lambda)<-1/2$. Hence for Re$(\lambda)<-1/2$, we see that 
	 \[
	 C_\psi f_\lambda(w)=\left(aw+b+\frac{b}{a-1}\right)^\lambda=\left(aw+\frac{ab}{a-1}\right)^\lambda=a^\lambda f_\lambda(w).
	 \]
	 This implies that each such $a^\lambda$ is an eigenvalue of infinite multiplicity since 
	 \[
	 C_\psi f_{\lambda+\frac{2\pi n }{\log a}i}=a^{\lambda+\frac{2\pi n }{\log a}i} f_{\lambda+\frac{2\pi n }{\log a}i}=a^{\lambda } f_{\lambda+\frac{2\pi n }{\log a}i}
	 \]
	 for each integer $n$. We note that in the proof of \cite[Theorem 7.4]{Eva Adjoints Dirichlet} it is claimed that $(z-b)^\lambda$ is an eigenvector for $C_\psi$ which is clearly incorrect. It follows that $C_\psi^*$ is not cyclic. But type ${\mathrm{I}}$ operators are scalar multiples of the adjoints of those of type ${\mathrm{II}}$. Therefore type ${\mathrm{I}}$ operators are \emph{not} cyclic.
	 \begin{prop}\label{type I cyclicity}If $\phi$ is a hyperbolic non-automorphism of type $\mathrm{I}$, then $C_\phi$ is not cyclic. 
	 	\end{prop}
	 
	 \subsection{Symbols of type $\mathrm{II}$} Let $\psi(w)=aw+b$ with $a\in(0,1)$ and Re$(b)>0$. Hence $\psi$ is of type $\mathrm{I}$. It is easy to see that $\psi$ has a fixed point $w=\frac{b}{1-a}$ which belongs to $\mathbb{C}_+$. Now since type $\mathrm{II}$ operators are scalar multiples of adjoints of type $\mathrm{I}$ operators, the following general result suffices to show that all type $\mathrm{II}$ operators are cyclic.
	 \begin{prop} Let $\psi$ be an analytic self-map of $\mathbb{C}_+$ with $\psi(\alpha)=\alpha$ for some $\alpha\in\mathbb{C_+}$ such that $C_\psi$ is bounded on $H^2(\mathbb{C}_+)$. Then $C^*_\psi$ is cyclic.
	 	\end{prop}
 	\begin{proof} We first note that $\psi$ cannot be an automorphism of $\mathbb{C}_+$. Indeed, if $\psi$ were an automorphism then we must have $\psi(w)=aw+ir$ with $a>0$ and $r\in\mathbb{R}$ and it is easy to see that $\psi$ has no fixed point when $a=1$ and the purely imaginary fixed point $\frac{ir}{1-a}$ when $a\neq 1$. Now let $\gamma(z)=\frac{1+z}{1-z}$ be the \emph{Cayley transform} of the open unit disk $\mathbb{D}$ onto $\mathbb{C}_+$. Hence $\Psi=\gamma^{-1}\circ\phi\circ\gamma$ is a non-automorphic self-map of $\mathbb{D}$ with an interior fixed point $\beta=\gamma^{-1}(\alpha)\in\mathbb{D}$. The well-known \emph{Denjoy-Wolff Theorem} says that the  composition iterates $\Psi^{[n]}\To\beta$ locally uniformly in $\mathbb{D}$ as $n\to\infty$. It also follows that the iterates $\Psi^{[n]}(z)$ are distinct for each $z\neq \beta$ (see \cite[Lemma 1]{Worner}).  Hence $\psi^{[n]}\To\alpha$ locally uniformly in $\mathbb{C}_+$ as $n\to\infty$ and $\psi^{[n]}(w)$ is a sequence of distinct points for each $w\neq\alpha$. Our goal is to prove that each reproducing kernel $k_w$ for $w\neq\alpha$ is a cyclic vector for $C^*_\psi$. Suppose $f\in H^2(\mathbb{C}_+)$ is orthogonal to $(C_\psi^*)^n k_w$ for all $n\in\mathbb{N}$. Then
 		\[
 		0=\left\langle f,(C_{\psi}^*)^nk_w\right\rangle =\left\langle C^n_{\psi}f,k_w\right\rangle =\left\langle f\circ \psi^{[n]},k_w\right\rangle =f(\psi^{[n]}(w))
 		\]
 		implies that $f$ vanishes on a sequence of distinct points with limit $\alpha$ in $\mathbb{C}_+$. Hence $f\equiv 0$ and $k_w$ is a cyclic vector for $C_\psi^*$ for each $w\neq\alpha$.
 		\end{proof}
 	This concludes the proof of the cyclicity of type $\mathrm{II}$ composition operators.
 	 \begin{cor}\label{type II cyclicity}If $\phi$ is a hyperbolic non-automorphism of type $\mathrm{II}$, then $C_\phi$ is cyclic.
 	\end{cor}
 	Therefore Proposition \ref{type I cyclicity} and Corollary \ref{type II cyclicity} imply that $C_\phi$ is \emph{not} complex symmetric when $\phi$ is a hyperbolic non-automorphism. We therefore obtain a characterization for complex symmetry of affine composition operators.
\begin{thm} Let $\phi(w)=aw+b$ be a self-map of $\mathbb{C}_+$. Then $C_\phi$ is complex symmetric on $H^2(\mathbb{C}_+)$ if and only if $a=1$ or Re$(b)=0$. That is, precisely when $C_\phi$ is normal. 
\end{thm}
\section{Cyclicity of $C_\phi$ }
The goal of this section is completely to characterize the cyclic $C_\phi$ on $H^2(\mathbb{C}_+)$. We saw in the previous section that hyperbolic non-automorphisms $\phi$ of type $\mathrm{I}$ and type $\mathrm{II}$ induce non-cyclic and cyclic $C_\phi$ respectively.  In contrast, the next result shows that \emph{all} parabolic non-automorphisms $\phi$ induce cyclic $C_\phi$ on $H^2(\mathbb{C}_+)$.

\begin{prop}\label{Parabolic cyclicity}Let $\phi(w)=w+b$ with $\mathrm{Re}(b)>0$. Then $C_{\phi}$ is cyclic on $H^2(\mathbb{C}_+).$
\end{prop}
\begin{proof} 
First note that the compositional iterates of $\phi$ are given by
$
\phi^{[n]}(w)=w+nb.
$
This implies that we have
\[
(C_{\phi}^nk_1)(w)=k_1(\phi^{[n]}(w))=k_1(w+nb)=\frac{1}{1+nb+w}=k_{b_n}(w)
\]
where $b_n=1+n\overline{b}$. Now if we assume some $f\in H^2(\mathbb{C}_+)$ is orthogonal to the span of the orbit $(C_\phi^n k_1)_{n\in\mathbb{N}}$, then
\begin{align*}
0=\left\langle f,C_{\phi}^nk_1\right\rangle =\left\langle f,k_{b_n}\right\rangle=f(b_n).
\end{align*}
To conclude the proof, it is enough to show that the sequence $(b_n)_{n\in\mathbb{N}}$ does not satisfy the \emph{Blaschke condition} for zeros of $H^2(\mathbb{C}_+)$ functions (see \cite[p. 53]{Garnett}). That is, we must prove that
\begin{equation}\label{Blaschke violates}
\sum_{n=1}^{\infty}\frac{\mathrm{Re}(b_n)}{1+\left| b_n\right|^2} =\infty.
\end{equation}
First note that 
\[
1+ \left| b_n\right|^2=1+(1+n\mathrm{Re}(b))^2+(n\mathrm{Im}(b))^2\leq 2(1+n\left| b\right|) ^2\leq 2(1+\left| b\right|)^2n^2
\]
which implies that
\[
\frac{\mathrm{Re}(b_n)}{1+\left| b_n\right|^2}\geq \frac{1+n\mathrm{Re}(b)}{2(1+\left| b\right|)^2n^2}\geq \frac{\mathrm{Re}(b)}{2(1+\left| b\right|)^2}\frac{1}{n}.
\]
Therefore \eqref{Blaschke violates} clearly holds and
hence $(b_n)_{n\in\mathbb{N}}$ cannot be a zero sequence for $f$ unless $f\equiv 0$.
\end{proof}

Hence the only case remaining is the cyclicity of $C_\phi$ where $\phi(w)=aw+b$ with $a>0$ and Re$(b)= 0$, that is precisely when $\phi$ is an automorphism of $\mathbb{C}_+$. This will be achieved with the help of the following result about the \emph{non-cyclicity} of certain multiplication operators on $L^2$ spaces of the real line. The idea of the proof is inspired by that of \cite[Theorem 3.13]{Eva role of spectrum}.

\begin{lem}\label{Non-cyclity of M on L^2} Suppose $s\in\mathbb{R}$ and let $M:=M_{e^{ist}}$ be the operator of multiplication by $e^{ist}$ on $L^2(\mathbb{R}^+,dt)$ or $L^2(\mathbb{R},dt)$. Then $M$ is not cyclic on either space.
	\end{lem}
\begin{proof}If $s=0$ then $M=I$ is clearly non-cyclic. So assume $s\not = 0$. Since $L^2(\mathbb{R}^+,dt)$ is clearly a reducing subspace for $M$ acting on $L^2(\mathbb{R},dt)$, it is enough to prove the result for $L^2(\mathbb{R}^+,dt)$. Consider any function $f
	\in L^2(\mathbb{R}^+,dt)$. Then we have 
	\[
	\mathrm{span}\{M^nf:n\in\mathbb{N}\}=\{pf: \textrm{where p is a polynomial in} \ e^{ist}\}.
	\]
	First suppose that $f$ vanishes on a set $A\subset\mathbb{R}^+$ of positive measure. Then each $pf$ vanishes on $A$ and hence
	sequences of these $pf$ can approximate only functions that vanish almost everywhere on $A$. Therefore $M$ is not cyclic in this case. For the other case, suppose $f\not\equiv 0$ on any set of positive measure. That $M$ is non-cyclic will follow from the fact that any polynomial in $e^{ist}$ is $2\pi/s$ periodic. 
	
	Let $\chi_{[0,1]}$ be the characteristic function of $[0,1]$ and suppose $p_nf\to\chi_{[0,1]}$ in $L^2(\mathbb{R}^+,dt)$ for some sequence $p_n$ of polynomials in $e^{ist}$. Then some subsequence $p_{n_k}f\to\chi_{[0,1]}$ pointwise almost everywhere. So on the one hand $p_{n_k}\to1/f$ almost everywhere on $[0,1]$ and $p_{n_k}\to 0$ almost everywhere on $(1,\infty)$ since $f\not = 0$ almost everywhere. But on the other hand the periodicity of $p_{n_k}$ implies that we also have $p_{n_k}\to 0$ almost everywhere on $[0,1]$ and hence that $1/f=0$ almost everywhere on $[0,1]$. This contradiction proves that $M$ is not cyclic in this case also. 
	\end{proof}

We are now ready to complete the characterization of linear fractional cyclicity. 

\begin{thm}Let $\phi(w)=aw+b$ be a self-map of $\mathbb{C}_+$. Then $C_\phi$ is cyclic on $H^2(\mathbb{C}_+)$ if and only if $\phi$ is a parabolic non-automorphism or a hyperbolic non-automorphism of type $\mathrm{II}$. That is precisely when $a\geq1$ and $\mathrm{Re}(b)>0$. 
	\end{thm}
\begin{proof}Since the only case that remains is when $\phi$ is an automorphism, we may assume $\mathrm{Re}(b)=0$. We first note that if $\Pi$ denotes the upper half-plane, then $(Uf)(w)=f(iw)$ defines a unitary map of the Hardy space $H^2(\Pi)$ of the upper half-plane onto $H^2(\mathbb{C}_+)$ and $C_\phi$ on $H^2(\mathbb{C}_+)$ is unitarily equivalent to $C_\psi$ on $H^2(\Pi)$ where $\psi(w)=aw+ib$ is a self map of $\Pi$. It follows from	Gallardo-Guti\'{e}rrez and Montes-Rodr\'{i}guez \cite[Theorem 7.1]{Eva Adjoints Dirichlet} that if $a=1$ then $C_\psi$ is similar to $M_{e^{-bt}}=M_{e^{-i\mathrm{Im}(b)t}}$ on $L^2(\mathbb{R}^+,dt)$, and hence so is $C_\phi$. Therefore Lemma \ref{Non-cyclity of M on L^2} shows that $C_\phi$ is not cyclic when $\phi$ is a parabolic automorphism. Similarly, when $\phi$ is a hyperbolic automorphism ($a\neq 1$),  then $C_\phi$ is similar to $M_{a^{-it-1/2}}=a^{-1/2}M_{e^{-it\log a}}$ on $L^2(\mathbb{R},dt)$ which is again not cyclic by Lemma \ref{Non-cyclity of M on L^2}. This completes the proof of the theorem.\end{proof}

\section{Hypercyclicity of $C_\phi$}
Finally we show that $H^2(\mathbb{C_+})$ does \emph{not} support any hypercyclic composition operator with affine and therefore linear fractional symbol. This is in sharp contrast to various weighted Hardy spaces of the open unit disk (see \cite[page 8]{Eva role of spectrum}).

\begin{thm} $C_\phi$ is not hypercyclic on $H^2(\mathbb{C_+})$ for any affine symbol $\phi$.
	\end{thm}
\begin{proof} 
	If $\phi$ is an automorphism or a parabolic non-automorphism, then $C_\phi$ is normal and hence is not hypercyclic (see \cite[Theorem 5.30]{Linear Chaos}). Similarly if $\phi$ is a hyperbolic non-automorphism of type $\mathrm{I}$, then $C_\phi$ is not cyclic by Proposition \ref{type I cyclicity} and hence is not hypercyclic either. The case that remains is when $\phi$ is hyperbolic non-automorphic of type $\mathrm{II}$, that is $\phi(w)=aw+b$ where $a>1$ and $\mathrm{Re}(b)>0$. By induction, it is easy to show that the $n$-th iterate of $\phi$ is given by
	\[
	\phi^{[n]}(w)=a^nw+\frac{(1-a^{n})b}{1-a}.
	\]
	 Then $|| C_{\phi}^n||_{H^2(\mathbb{C}_+)}=\sqrt{{\phi^{[n]}}'(\infty)}=\sqrt{1/a^n}$ (see \eqref{Ang. Der. at infinity} and \cite{Eliot Jury}). Since $a\in (1,\infty)$, the sequence $|| C_{\phi}^n||_{H^2(\mathbb{C}_+)}\to 0$ as $n\to\infty$. This implies that $C_{\phi}$ is not hypercyclic.
\end{proof}
\section*{Acknowledgement}
	This work constitutes a part of the doctoral dissertation of the second author. The first author is partially supported by a FAPESP grant (17/09333-3).  Both authors thank the anonymous referee for the many comments and suggestions to improve this article.
\bibliographystyle{amsplain}

\end{document}